\documentclass[11pt]{amsart}
\usepackage[margin=1in]{geometry}

\usepackage{amssymb}
\usepackage{amsthm}
\usepackage{amsmath}
\usepackage{mathrsfs}
\usepackage{amsbsy}
\usepackage[all]{xy}
\usepackage{bm}
\usepackage{hyperref}
\usepackage{tikz}
\usepackage{array}
\usepackage{float}
\usepackage{enumerate}
\usepackage{xcolor}
\usepackage{hhline}
\allowdisplaybreaks
\usepackage[noadjust]{cite}

\usepackage{caption}
\usepackage{tabu}
\usepackage{diagbox}

\usepackage[noabbrev,capitalise,nameinlink]{cleveref}

\definecolor{lavender}{rgb}{0.5,0,1.0}

\hypersetup{colorlinks=true, citecolor=lavender, linkcolor=lavender,urlcolor=lavender}

\newenvironment{enumerate*}%
  {\begin{enumerate}[(I)]%
    \setlength{\itemsep}{10pt}%
    \setlength{\parskip}{0pt}}%
  {\end{enumerate}}

\newtheorem{theorem}{Theorem}[section]
\newtheorem{proposition}[theorem]{Proposition}

\newtheorem{conjecture}[theorem]{Conjecture}
\newtheorem{question}[theorem]{Question}

\newtheorem{lemma}[theorem]{Lemma}

\theoremstyle{definition}

\newcommand{\dfn}[1]{\textcolor{blue}{\emph{#1}}}

\newcommand{\SortNoop}[1]{}

\usepackage{doi}

\def\Z{\mathbb{Z}}
\newcommand{\ba}{{\rm \bf{a}}}
\newcommand{\bb}{{\rm \bf{b}}}
\newcommand{\bc}{{\rm \bf{c}}}
\renewcommand{\Re}{\operatorname{Re}}

\begin{document}

\title{Tilings of Benzels via Generalized Compression}
\subjclass[2010]{}

\author[Colin Defant]{Colin Defant}
\address[]{Department of Mathematics, Massachusetts Institute of Technology, Cambridge, MA 02139, USA}
\email{colindefant@gmail.com}

\author[Leigh Foster]{Leigh Foster}
\address[]{Department of Mathematics, University of Oregon, Eugene, OR 97403, USA}
\email{leighf@uoregon.edu}

\author[Rupert Li]{Rupert Li}
\address[]{Massachusetts Institute of Technology, Cambridge, MA 02139, USA}
\email{rupertli@mit.edu}

\author[James Propp]{James Propp}
\address[]{Department of Mathematical Sciences, UMass Lowell, Lowell, MA 01854, USA}
\email{jamespropp@gmail.com}

\author[Benjamin Young]{Benjamin Young}
\address[]{Department of Mathematics, University of Oregon, Eugene, OR 97403, USA}
\email{bjy@uoregon.edu}

\maketitle

\begin{abstract}
Defant, Li, Propp, and Young recently resolved two enumerative conjectures of Propp concerning the tilings of regions in the hexagonal grid called \textit{benzels} using two types of prototiles called \textit{stones} and \textit{bones} (with varying constraints on allowed orientations of the tiles).
Their primary tool, a bijection called \emph{compression} that converts certain $k$-ribbon tilings to $(k-1)$-ribbon tilings, allowed them to reduce their problems to the enumeration of dimers (i.e., perfect matchings) of certain graphs.
We present a generalized version of compression that no longer relies on the perspective of partitions and skew shapes.
Using this strengthened tool, we resolve three more of Propp's conjectures and recast several others as problems about perfect matchings.
\end{abstract}

\section{Introduction}\label{SecIntro}
Enumeration of tilings of a region using translates of prescribed \textit{prototiles} is a well-established topic in combinatorics; one notable example is the problem of counting domino tilings of the Aztec diamond in the square grid \cite{EKLP}, which is the setting in which the arctic circle phenomenon was first discovered \cite{CEP}. We refer readers to \cite{ProppSurvey} for a survey of enumerative tiling methods and results. Conway and Lagarias \cite{ConwayLagarias} studied some tiling problems in the hexagonal grid, where both the regions to be tiled and the tiles themselves are composed of regular hexagons and are sometimes referred to as \dfn{polyhexes}.
In their setting, the tiles are \dfn{trihexes}, which consist of three hexagons -- specifically, these are the trihexes referred to below as stones and bones.
In a novel application of combinatorial group theory, Conway and Lagarias devised a new necessary condition for a region to admit a tiling by stones and bones. Thurston \cite{Thurston} expanded upon these results with alternative perspectives, and Lagarias and Romano \cite{LagariasRomano} proved an exact formula for the number of tilings in a particular one-parameter family of trihex tiling problems.
Meanwhile, physicists working in an essentially equivalent (dual) setting studied \emph{trimer covers} \cite{VN} of the regular 6-valent planar graph. However, the physicists' interests were in asymptotics rather than exact enumerations, and their proofs relied on the Bethe ansatz, which has not been rigorously established in all the contexts where it has been applied (even though it tends to give correct answers).

Motivated by finding a trimer analogue of the Aztec diamond (and perhaps new kinds of arctic circle phenomena), Propp \cite{Propp2022trimer} proposed using the tiles studied by Conway and Lagarias to tile different sorts of regions in the hexagonal grid not considered in earlier literature, which he dubbed {\em benzels}.
He made numerous conjectures regarding the exact number of tilings of those regions, and Defant, Li, Propp, and Young \cite{DLPY} resolved some of his conjectures.
The key mechanism they developed is \textit{compression}, which converts the problem from one involving 3-cell tiles to one involving 2-cell tiles. Tilings using 2-cell tiles are equivalent to perfect matchings of graphs, which have been well-studied and are amenable to a host of enumerative techniques.

We follow the conventions of \cite{DLPY}, which expand upon the conventions of \cite{Propp2022trimer}.
We view the hexagonal grid as a tiling of the complex plane with regular hexagons of side length 1, which we refer to as \dfn{cells} of the grid.
Specifically, one of the hexagons has vertices $1\pm\omega^j$ for $j\in\{0,1,2\}$, where $\omega=e^{2\pi i/3}$ is a primitive third root of unity.
Suppose $a$ and $b$ are positive integers satisfying $2 \leq a \leq 2b$ and $2\leq b\leq 2a$; we define the \dfn{$(a,b)$-benzel} to be the union of the cells lying entirely within the hexagon with vertices $\omega^j(a\omega+b)$ and $-\omega^j(a+b\omega)$ for $j\in\{0,1,2\}$.
This hexagon is invariant under rotation by 120$^\circ$ and has sides whose lengths alternate between $2a-b$ and $2b-a$ (hence the inequalities $a\leq 2b$ and $b\leq 2a$).
In particular, the equality case corresponds to a triangle, a degenerate hexagon.
As discussed in \cite[Remark 5.1]{DLPY}, the $(n,2n-2)$-benzel, the $(n,2n-1)$-benzel, and the $(n,2n)$-benzel coincide because the extension of the boundary hexagon from the case $(n,2n-2)$ to the case $(n,2n)$ does not add any additional cells in their entirety.
The same holds for the $(2n-2,n)$-benzel, the $(2n-1,n)$-benzel, and the $(2n,n)$-benzel.
Aside from these cases, the map sending the pair $(a,b)$ to the $(a,b)$-benzel is injective, so henceforth we will assume the stronger inequalities $2 \leq a \leq 2b-2$ and $2\leq b\leq 2a-2$.
\cref{fig:benzel_example} shows the $(9,11)$-benzel.
\begin{figure}
\centering
\includegraphics[width = 0.33\textwidth]{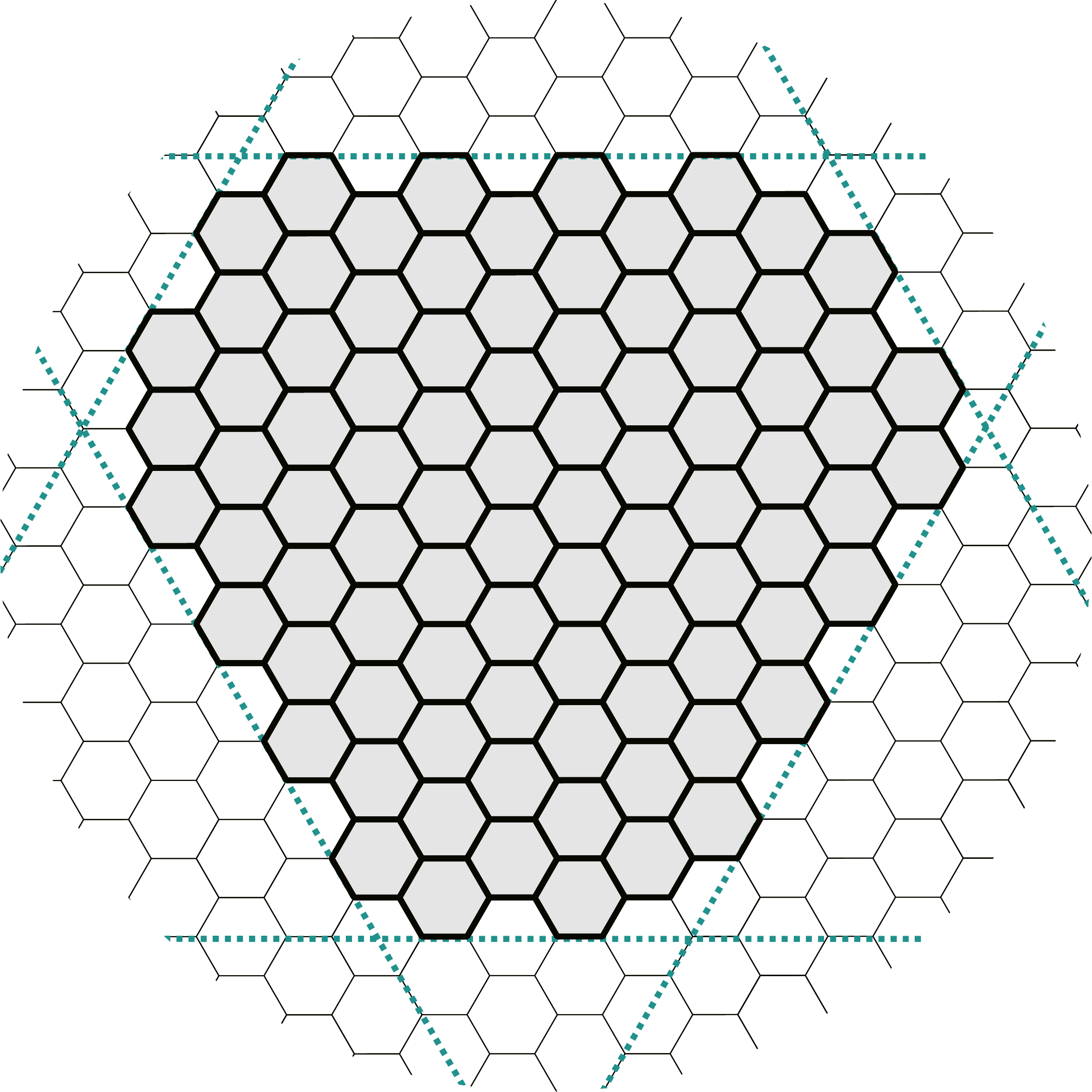}
\caption{The $(9,11)$-benzel (shaded).}
\label{fig:benzel_example}
\end{figure}

The prototiles in \cite{ConwayLagarias,Propp2022trimer} are special types of trihexes. These prototiles come in two forms: a \dfn{stone} consists of three pairwise adjacent cells arranged in a triangle, while a \dfn{bone} consists of three consecutive cells whose centers are colinear.
Here, we follow the naming conventions of Propp \cite{Propp2022trimer}; Conway and Lagarias \cite{ConwayLagarias} and Thurston \cite{Thurston} used different terminology.
It will be useful later to distinguish between the different rotations of these stones and bones.
As a result, we consider the five prototiles in \cref{prototiles} to be distinct, calling them the \dfn{left stone}, \dfn{right stone}, \dfn{rising bone}, \dfn{falling bone}, and \dfn{vertical bone}, respectively.
Tilings of regions in the hexagonal grid using only these five prototiles will be referred to as \dfn{stones-and-bones tilings}.

\begin{figure}[htbp]

\includegraphics[width = 0.75\textwidth]{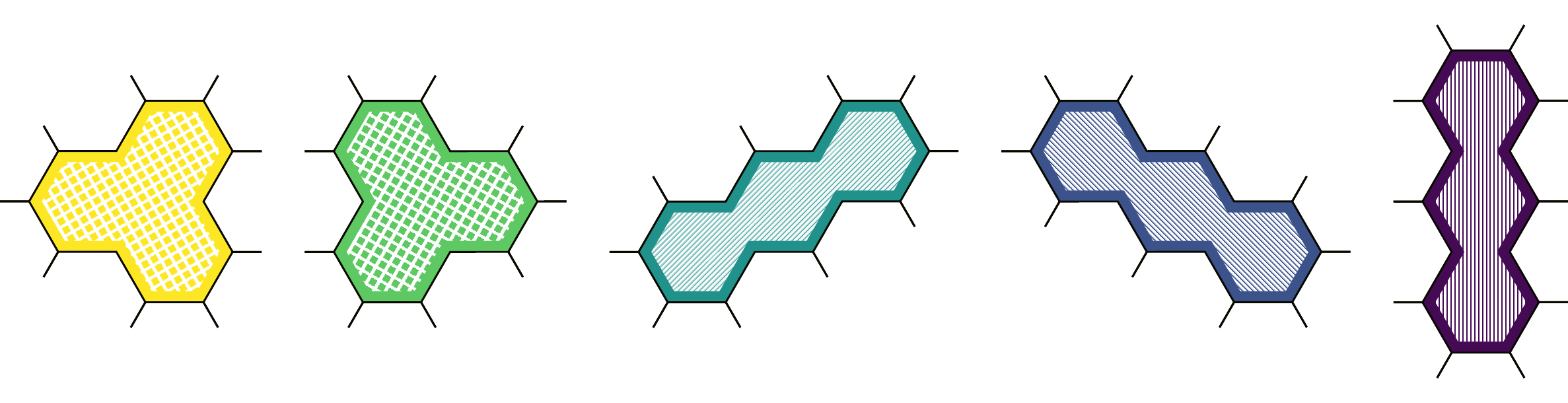}
\caption{The five prototiles: the left stone, right stone, rising bone, falling bone, and vertical bone, respectively.}
\label{prototiles}
\end{figure}

Tiling problems frequently restrict the set of possible prototiles.
Using our five prototiles, there are $2^5-1=31$ nonempty subsets and, consequently, 31 enumeration questions that we can ask for each $(a,b)$-benzel.
However, as benzels have $120^\circ$ rotational symmetry that preserves the two stones, the exact set of allowed bones does not affect the number of stones-and-bones tilings; all that matters is the number of allowed types of bones.
This equivalence reduces the number of problems to $2\cdot2\cdot4-1=15$.
Propp \cite{Propp2022trimer} computed numerical results for these various problems and provided a collection of conjectures and open questions.

Note that the $(a,b)$-benzel and the $(b,a)$-benzel are reflections of each other across the real axis; as reflection preserves the two stones and permutes the three bones, each tiling problem is the same for the $(a,b)$-benzel as for the $(b,a)$-benzel.
We will frequently use this symmetry to assume without loss of generality that $a\leq b$.

As discussed in \cref{SecPrelim}, our framework is best suited to enumeration problems forbidding at least one type of bone, so without loss of generality, the vertical bone is forbidden.
Defant, Li, Propp, and Young \cite[Theorem 1.2]{DLPY} enumerated the stones-and-bones tilings of benzels that use only left stones, rising bones, and falling bones.
Such stones-and-bones tilings exist if and only if $(a,b)\in\{(3n,3n),(3n+1,3n+2),(3n+2,3n+1)\}$ for some positive integer $n$; in this case, the number of such tilings is $(2n)!!$.

Defant, Li, Propp, and Young \cite[Theorem 1.3]{DLPY} also enumerated the stones-and-bones tilings of benzels that use only right stones, rising bones, and falling bones when $a+b\not\equiv2\pmod{3}$; in this case, at most 1 such tiling exists.
They were unable to solve the problem when  ${a+b\equiv2\pmod{3}}$; in this case, Propp \cite[Problem 5]{Propp2022trimer} gave a conjectural formula. This conjectural formula is significantly more complicated than the solutions to the aforementioned enumeration problems, suggesting the difficulty of the problem.
We prove this formula, stated in the following theorem.
\begin{theorem}\label{theorem:problem_5}
    Let $a$ and $b$ be integers with $2\leq a \leq b \leq 2a$ and $a+b\equiv2\pmod{3}$.
    Let $n=b+1-a$ and $k=(2a-b-1)/3$ so that $(a,b)=(n+3k,2n+3k-1)$.
    Then the $(a,b)$-benzel has
    \begin{align*}
        \prod_{j=1}^k \frac{(2j)!(2j+2n-2)!}{(j+n-1)!(j+n+k-1)!}
    \end{align*}
    tilings by right stones, rising bones, and falling bones.
\end{theorem}

This result, combined with the results of Defant, Li, Propp, and Young \cite{DLPY}, completely enumerates the stones-and-bones tilings of benzels that use only right stones, rising bones, and falling bones.

When $a+b\equiv1\pmod{3}$, \cite[Theorem 1.1]{DLPY} implies that exactly one stones-and-bones tiling of the $(a,b)$-benzel exists, consisting entirely of right stones. However, aside from this case, nothing is known about the number of stones-and-bones tilings of benzels that use left stones, right stones, rising bones, and falling bones (i.e., all stones and bones except the vertical bone).
Propp \cite{Propp2022trimer} stated a number of open questions concerning such enumerations (Problems 8 through 13), and we solve Problems~12 and~13 concerning the $(n+2,2n+1)$-benzel and the $(n+2,2n)$-benzel in \cref{proposition:problem_12,proposition:problem_13}, respectively.

In \cref{SecPrelim}, we follow the procedure of Defant, Li, Propp, and Young \cite{DLPY} to convert benzel tilings to ribbon-tilings (more specifically 3-ribbon tilings), which have been studied by others (see for instance \cite{Pak}).
In \cref{SecCompress}, we establish the technique of compression in greater generality than in \cite{DLPY}.
In \cref{SubSecAztecTriangle}, as a warmup, we apply our more generalized compression technique, which converts 3-ribbon tilings into 2-ribbon tilings, to the $(n+2,2n+1)$-benzel and the $(n+2,2n)$-benzel, solving Problems 12 and 13 of \cite{Propp2022trimer}.
We use our compression technique to prove \cref{theorem:problem_5} in \cref{SecDecomposition}.
Finally, in \cref{SecConclusion}, we restate Problems~8 through~11 of \cite{Propp2022trimer} in their compressed forms and leave them as open questions, sharpening Problems~9 and~11 in the process.

\section{Preliminaries}\label{SecPrelim}
Following the conventions of \cite{DLPY}, we draw the square grid in the complex plane, where each square grid cell has center of the form $z=x+iy\in\Z[i]$ with $x+y$ odd, and the cell centered at $z$ has vertices of the form $z+i^j$ for $j\in\{0,1,2,3\}$.
To avoid confusion, hexagonal cells will continue to be called \emph{cells}, while square cells will be called \dfn{boxes}.
We refer to a specific box by identifying it with its center.
To leverage the existing literature for tilings of regions in the square grid, we follow the ``squarification'' procedure of \cite{DLPY} to convert tilings in the hexagonal grid to tilings in this square grid.
Note that the unit-width vertical strips $\{z\in\mathbb C:\frac{n-2}{2} \leq \Re z < \frac{n}{2}\}$ for $n \in 3\Z$ are traversed only by horizontal edges of the hexagonal grid.
See \cref{squarify} for an example of the squarification process applied to the $(9,11)$-benzel; the horizontal edges, contained in the aforementioned unit-width vertical strips, are highlighted in green in the leftmost pane.
\begin{figure}[htbp]
    \centering
    \includegraphics[width = \textwidth]{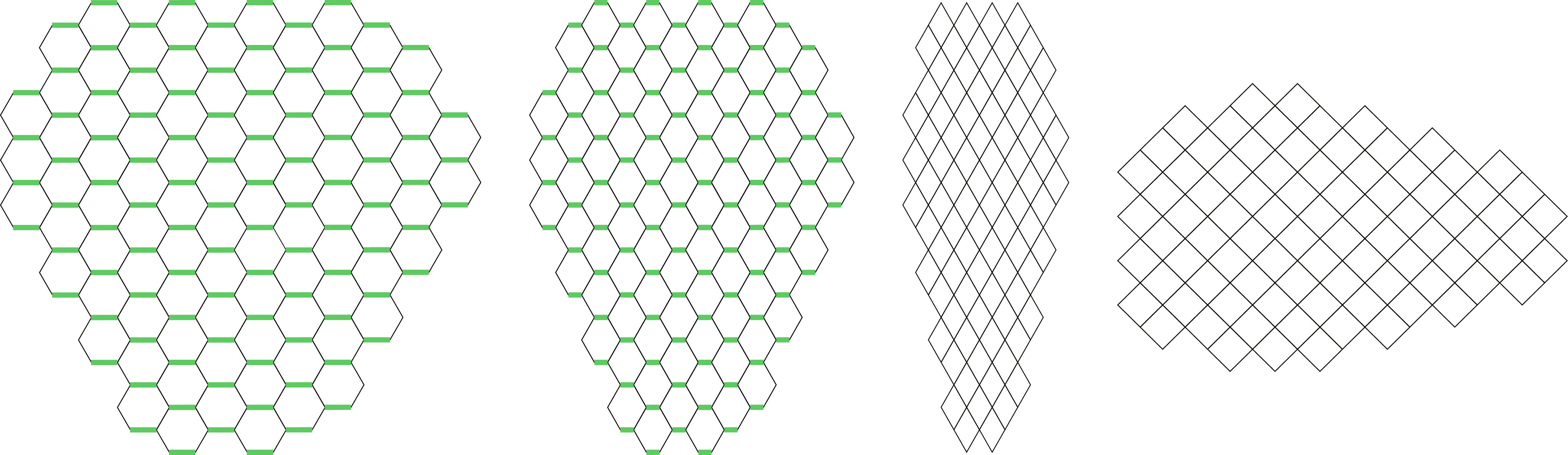}
    \caption{The squarification process applied to the (9,11)-benzel.}
    \label{squarify}
\end{figure}
Removing these strips and compressing the complex plane appropriately yields a bijection mapping hexagons to rhombuses in the resulting rhombic grid; see the third pane from the left in \cref{squarify}, where the second pane illustrates an intermediate state in this step.
Rotating this grid by 90$^\circ$ counterclockwise and rescaling the axes suitably transforms this rhombic grid to our aforementioned square grid.
Boxes with the same real part form a {\em column}.

The four prototiles other than the vertical bone (which we have excluded) appear in \cref{prototiles}; these prototiles become the four prototiles in \cref{square tiles} under squarification and rotation.
Because of the rotation process, the names of the prototiles are no longer well-suited; to resolve this issue, the authors of \cite{DLPY} introduced new names for the prototiles in the rotated setting: we rename the right stone, left stone, rising bone, and falling bone the \dfn{mountain stone}, \dfn{valley stone}, \dfn{negative bone}, and \dfn{positive bone}, respectively.
Recall that we excluded the vertical bone; this ensures that all squarified tiles are ribbon tiles, although our grid is rotated by $45^\circ$ relative to the usual pictures that appear in the literature on ribbon tilings.
Within our rotated square grid framework, a \dfn{ribbon} is a connected union of boxes in the square grid occupying consecutive columns, and a \dfn{$k$-ribbon} is a ribbon consisting of exactly $k$ boxes.
Our four remaining prototiles are precisely the four 3-ribbons.
Tilings of regions in the square grid using ribbons are referred to as \dfn{ribbon tilings}, and ribbon tilings only using $k$-ribbons are \dfn{$k$-ribbon tilings}.
\begin{figure}[htbp]
    \centering
    \includegraphics[width = 0.75\textwidth]{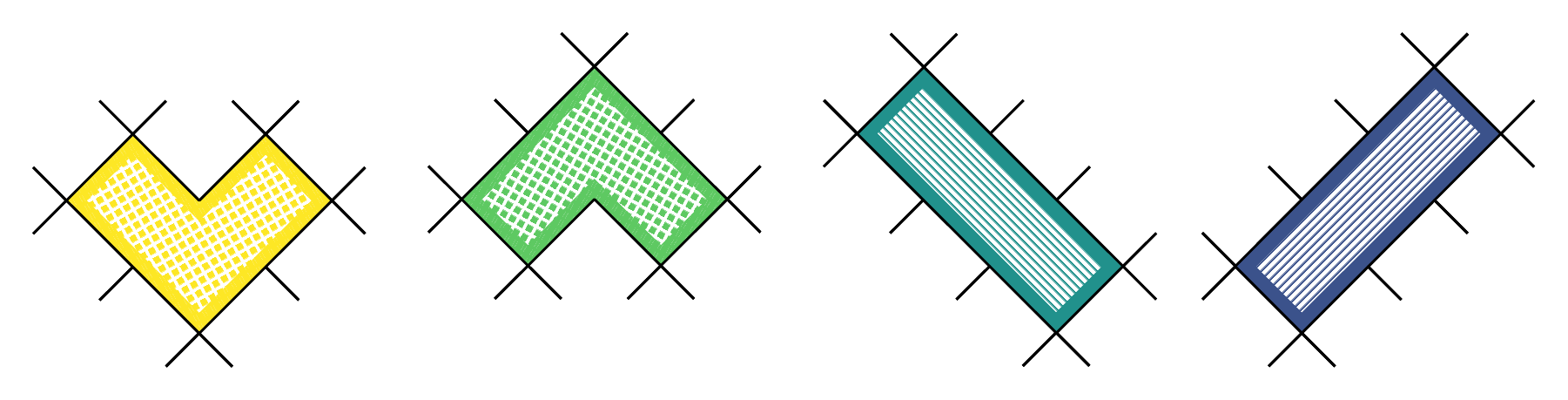}
    \caption{The four squarified tiles after $90^\circ$ rotation: the valley stone, mountain stone, negative bone, and positive bone.}
    \label{square tiles}
\end{figure}

An additional tool is the following factorization lemma, which we will use after applying the compression technique.
Recall that compression converts 3-ribbon tilings into 2-ribbon tilings, which are in bijection with perfect matchings of a graph.

\begin{lemma}\label{lemma:color_decomposition}
Let $G$ be a bipartite graph with vertex set $V$, with vertices colored black and white in such a way that every edge of $G$ joins vertices of opposite colors. Suppose also that $V = A \sqcup B$, where $A$ contains as many white vertices as black vertices and all edges joining $A$ and $B$ connect a white vertex in $A$ to a black vertex in $B$.
Then no perfect matching of $G$ contains an edge joining $A$ and $B$, so the number of perfect matchings of $G$ equals the product of the number of perfect matchings of the induced subgraph on $A$ and the number of perfect matchings of the induced subgraph on $B$.
\end{lemma}
\begin{proof}
Every perfect matching of $G$ contains (up to) four kinds of edges, according to whether the white vertex lies in $A$ or $B$ and according to whether the black vertex lies in $A$ or $B$. 
Since $A$ has as many white vertices as black vertices, the number of edges joining a white vertex in $A$ to a black vertex in $B$ must equal the number of edges joining a black vertex in $A$ to a white vertex in $B$. But by assumption, the latter number is zero, so the former number must be zero as well.  
\end{proof}

\cref{lemma:color_decomposition} is a special case of a more general result that applies with $k$ color classes and $k$-ribbon tilings; for instance, the case $k=3$ appears in~\cite{DLPY}, from which the generalization to other $k$ is apparent.

\section{Compression}\label{SecCompress}
Compression is a streamlined version of the combinatorial transformation
that in~\cite{DLPY} was described in terms of the abacus bijection of Gordon James.
Under that earlier approach, we take the region we wish to tile by $k$-ribbons,
view it as a Young diagram, and disassemble the diagram into a core
and a quotient (a $k$-tuple of partitions, one of which turns out to be empty);
we then remove the empty partition from the quotient and obtain a
$(k-1)$-quotient that can be used to assemble a new Young diagram
whose $(k-1)$-ribbon tilings are in bijection with $k$-ribbon tilings of
the original Young diagram. Care is required when the core
of the original Young diagram is nonempty.

Our new approach does not involve cores or quotients. 
Note that the squarification process of \cref{squarify} results in a row-convex and column-convex polyomino that has been rotated by $45^\circ$.
Let $S$ denote such a rotated shape, 
which is to be tiled by $k$-ribbon tiles so that each $k$-ribbon
consists of boxes from $k$ consecutive columns. 
Let $r$ be the number of columns of $S$ containing one or more boxes, and index these columns as $1,\ldots,r$ from left to right.
Let $m_i$ denote the number of boxes in the $i$-th column of $S$, and let $M(x)=\sum_i m_i x^i$ be the corresponding generating function.
For each $k$-ribbon tiling $T$ of $S$ and for each 
$i\in\{1,\ldots,r\}$, let $p_i (T)$ be the number of $k$-ribbon tiles 
in $T$ whose leftmost box is in column $i$,
and let $P(x)=\sum_i p_i(T) x^i$ be the corresponding generating function. 

Every box in the $i$-th column belongs to a tile of $T$ whose leftmost
box belongs to some column between the $(i-k+1)$-th and the $i$-th, inclusive, so
$M(x) = (1+x+\dots+x^{k-1}) P(x)$. Since the polynomial $M(x)$ does
not depend on $T$, neither does $P(x)$,
so the coefficients $p_i (T)$ do not depend on $T$ either. Thus, we will henceforth write $p_i$ instead of $p_i(T)$. 

Arguments similar to those employed in the preceding paragraph were originally proposed by Pak, though without the use of generating functions;
see the paragraph that immediately follows the proof of Theorem 8.2 in~\cite{Pak}.
The requirement that $M(x) / (1+x+\dots+x^{k-1})$ be a
polynomial with nonnegative coefficients provides a criterion
for tileability that is stronger than the usual coloring
argument; for example, the region shown in \cref{fig:hexomino}
\begin{figure}[htbp]
    \centering
    \includegraphics[width=0.1124\linewidth]{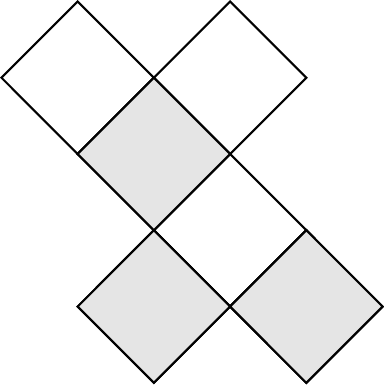}
    \caption{A color-balanced region that cannot be tiled by 2-ribbons.}
    \label{fig:hexomino}
\end{figure}
has equal numbers of shaded and unshaded boxes,
but since $(2x^1 + 1x^2 + 1x^3 + 2x^4) / (1 + x)$ equals $2x^1 - x^2 + 2x^3$
(a polynomial with a negative coefficient), the region does not admit a 2-ribbon (i.e., domino) tiling.

Our shape $S$ is vertically convex so that
all the boxes in the $i$-th column are corner-to-corner contiguous.
If $p_{i-k}$ and $p_{i}$ both vanish, then there does not exist a $k$-ribbon tile
whose rightmost box is in the $(i-1)$-th column or whose leftmost box is in
the $i$-th column. This implies that $m_{i-1} = m_{i}$, but it also implies more: if $T$ is a tiling of $S$ and $1\leq j\leq m_i$, then the tile of $T$ that contains the $j$-th box in column $i-1$ must also contain the $j$-th box in
column $i$ (where boxes in a column are indexed from top to bottom).

\begin{proposition}
\label{prop:compression}
Let $S$ be a horizontally and vertically convex shape, with the quantities $m_i$ and $p_i$ as defined above. 
Suppose there is an integer $i_0$ such that $p_{i}=0$ for all $i\equiv i_0\pmod{k}$. 
Define a smaller (``compressed'') shape $S'$ by identifying the $m_i$ boxes in column $i$ with the respective $m_i$ boxes in column $i-1$ from top to bottom, for every $i\equiv i_0\pmod{k}$.
Then the $k$-ribbon tilings of $S$ are in bijection with the $(k-1)$-ribbon tilings of $S'$.
\end{proposition}
\begin{proof}
Every $k$-ribbon tile must contain paired boxes from two paired columns, and by identifying the two boxes, we turn the $k$-ribbon tile into a $(k-1)$-ribbon tile.
\end{proof}

Before returning to benzels, we illustrate the procedure with an example taken from an article by Chen and Kargin~\cite{ChenKargin}.
The left part of \cref{fig:ck} shows (a rotated version of) Figure 3 from their article.
We have \[M(x)=3 x^1 + 3 x^2 + 4 x^3 + 3 x^4 + 3 x^5 + 4 x^6 + 3 x^7 + 3 x^8 + 4 x^9 + 3 x^{10} + 3 x^{11},\]
and the polynomial \[P(x) = 3 x^1 + 1 x^3 + 2 x^4 + 2 x^6 + 1 x^7 + 3 x^9\]
does not have any terms with exponents congruent to 2 modulo~$3$. Hence we may identify the boxes in columns 1 and 2, columns 4 and 5, columns 7 and 8, and columns 10 and 11 using the ``sutures'' that are shown,
obtaining the region shown in the right part of \cref{fig:ck}.
\begin{figure}[h!]
\centering
\includegraphics[width=0.8\textwidth]{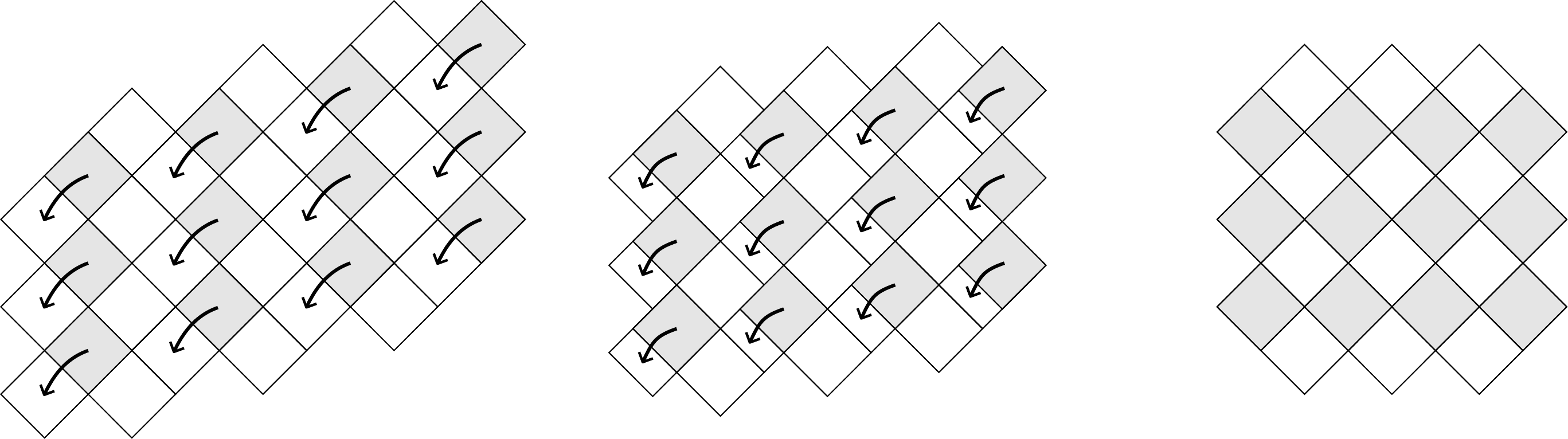}
\caption{Compressing a region into an Aztec diamond.}
\label{fig:ck}
\end{figure}
Therefore the 3-ribbon tilings of the region ($S$) on the left
are equinumerous with the 2-ribbon tilings of the region ($S'$) on the right, which is the Aztec diamond of order 3.
Indeed,~\cite[Theorem~2.5]{ChenKargin} can be proved using compression.
(The bijection used in Chen and Kargin's proof appears to be the same as ours, but we have not verified this.)

We now consider the region obtained from the $(a,b)$-benzel
via contraction of horizontal edges, squarification, and rotation, as described in \cref{SecPrelim}.
When $S$ is this region obtained from the $(a,b)$-benzel, write the polynomials $M(x)$ and $P(x)$ as $M_{a,b}(x)$ and $P_{a,b}(x)$, respectively.
Then $M_{a,b}(x)=(1+x+x^2)P_{a,b}(x)$,
where the polynomial $P_{a,b}(x)$ has a structure
that depends in a simple way on the parameters $a$ and $b$,
with different behavior according to $a+b$ (mod 3).
For $1 \leq i \leq k$, let \[Q_{i,k}(x)=\sum_{m=0}^{i-1}(m+1)x^m+\sum_{m=i}^{k-1}ix^m=1 x^0 + 2 x^1 + 3 x^2 + \dots + i x^{i-1} + i x^{i} + \dots + i x^{k-1}\]
and let $R_{i,k}(x)=x^{k-1}Q_{i,k}(1/x)$ be the polynomial of degree $k-1$ obtained from $Q_{i,k}(x)$ by reversing the order of the coefficients. For example, we have \[Q_{3,5}(x) = 1 + 2x + 3x^2 + 3x^3 + 3x^4\quad\text{and}\quad R_{3,5}(x) = 3 + 3x + 3x^2 + 2x^3 + x^4.\]
Also, for $1 \leq i,j \leq k$, let 
\begin{align*}
T_{i,j,k}(x)&=\sum_{m=0}^{k-i}(i+m)x^m+\sum_{m=k-i+1}^{2k-i-j}(2k-i-m)x^m \\ 
&=i x^0 + (i+1) x^1 + \dots + (k-1) x^{k-i-1} + k x^{k-i} + (k-1) x^{k-i+1} + \dots + j x^{2k-i-j}.
\end{align*}
For example, \[T_{1,2,4}(x) = 1 + 2x + 3x^2 + 4x^3 + 3x^4 + 2x^5.\]

\begin{lemma}
\label{lemma-polynomial}
Fix integers $a,b \geq 1$ with $a \leq 2b$ and $b \leq 2a$. Let $\gamma$ be the unique element of $\{1,2,3\}$ such that $a+b\equiv\gamma\pmod 3$, and let $s=(2a-b-3+\gamma)/3$ and $t=(2b-a-3+\gamma)/3$. Then 
\[P_{a,b}(x)=\begin{cases}
        x^3 T_{s+1,t+1,(a+b-1)/3}(x^3) & \text{if }\gamma=1, \\
        x R_{s,(a+b-2)/3}(x^3) + x^2 Q_{t,(a+b-2)/3}(x^3) & \text{if }\gamma=2, \\ 
        x R_{s,(a+b-3)/3}(x^3) + x^3 Q_{t,(a+b-3)/3}(x^3) & \text{if }\gamma=3.
    \end{cases}\]
\end{lemma}

\noindent
(The lemma contains more information than is required for this article; in order to apply compression, all we need to know is the existence of suitable polynomials $Q$, $R$, and $T$, not their precise forms. However, we include the formulas here as they might be useful to future researchers.)

\begin{proof}
Let us assume $\gamma=3$; a similar argument handles the other two cases. Recall (see \cite{KimPropp}) that the boundary of the $(a,b)$-benzel is given by the word
$$(\bb,\ba',\bb,\bc')^s (\bc,\ba',\bb,\ba')^t
  (\bc,\bb',\bc,\ba')^s (\ba,\bb',\bc,\bb')^t
  (\ba,\bc',\ba,\bb')^s (\bb,\bc',\ba,\bc')^t$$
where $\ba,\bb,\bc,\ba'\bb'\bc'$ are the unit vectors
pointing from $0$ to $1,\omega,\omega^2,-1,-\omega,-\omega^2$, respectively
(where $\omega$ is a primitive 3rd root of unity as earlier),
with $s = (2a-b)/3$ and $t = (2b-a)/3$.
When we apply the compression, squarification, and rotation operations described in \cref{SecPrelim},
two of these vectors shrink away and the remaining four change,
but the combinatorial description remains the same,
so the method
used to prove \cite[Theorem~1]{KimPropp} applies here as well.
\end{proof}

We are mostly interested in the case in which $a+b \equiv 2\pmod{3}$, but remarks on the other two cases are in order.

When $a+b \equiv 1 \pmod{3}$, the coefficient of $x^i$ in $P(x)$ vanishes for every $i \equiv 1,2 \pmod{3}$.
In this case, the bijection in \cref{prop:compression}
can be used twice to let us put 3-ribbon tilings
of the region in question in bijection with
1-ribbon tilings of a reduced region,
but every region has exactly one 1-ribbon tiling
(and in fact the uncompressed tiling is composed exclusively of right stones).

The case in which $a+b \equiv 0 \pmod{3}$ is subtler.
Here the coefficient of $x^i$ in $P(x)$ vanishes
whenever $i \equiv 2$ (mod 3),
so we can apply compression (once) to put 3-ribbon tilings
of the region in question in bijection with
2-ribbon tilings of a reduced region.
These 2-ribbon tilings correspond to perfect matchings
of graphs like the one shown in \cref{fig:0mod3}.
Here we set $s=(2a-b)/3$ and $t=(2b-a)/3$.
(We have flipped the graph across a horizontal axis 
in \cref{fig:0mod3} so as to be more consistent with
\Cref{fig:2mod3}, the analog for the case $a+b\equiv2\pmod{3}$;
the flip of course has no effect on the number of perfect matchings.)
Edges colored teal correspond to right/mountain stones, edges colored yellow correspond to left/valley stones,
and black edges correspond to bones.
Thus, forbidding right (respectively, left) stones in the stones-and-bones tiling of an $(a,b)$-benzel
corresponds to forbidding teal (respectively, yellow) edges in the associated perfect matching.

\begin{figure}[h!]
\centering
\includegraphics[width=0.3\linewidth]{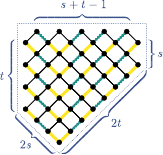}
\caption{The graph dual to the compressed $(8,10)$-benzel. In this example, we have $s=2$ and $t=4$.}
\label{fig:0mod3}
\end{figure}

Finally, we have the case in which $a+b \equiv 2\pmod{3}$, which is relevant to \cref{theorem:problem_5}.
Here, the coefficient of $x^i$ in $P(x)$ vanishes
whenever $i \equiv 0$ (mod 3),
so we can apply compression (once) to put 3-ribbon tilings
of the region in question in bijection with
2-ribbon tilings of a reduced region.
These 2-ribbon tilings correspond to perfect matchings
of graphs like the one shown in \cref{fig:2mod3},
where the ``sutures'' of this compression process,
along with an example 3-ribbon tiling that compresses to a 2-ribbon tiling,
are shown in \cref{fig:suture_2mod3}.
Here, we set $s=(2a-b-1)/3$ and $t=(2b-a-1)/3$.
As in the previous case,
edges colored teal correspond to right stones
while edges colored yellow correspond to left stones.

\begin{figure}[htbp]
\centering
\includegraphics[width=0.3\linewidth]{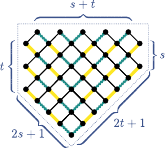}
\caption{The graph dual to the compressed $(8,9)$-benzel. In this example, we have $s=2$ and $t=3$.}
\label{fig:2mod3}
\end{figure}

\begin{figure}[htbp]
    \centering
    \includegraphics[width=0.9\textwidth]{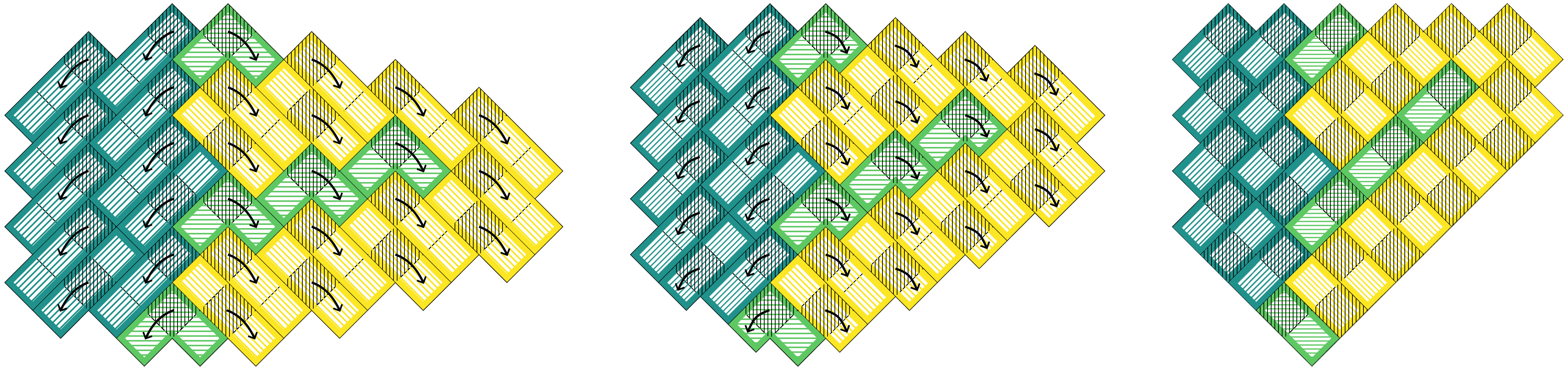}
    \includegraphics[width=0.28\textwidth]{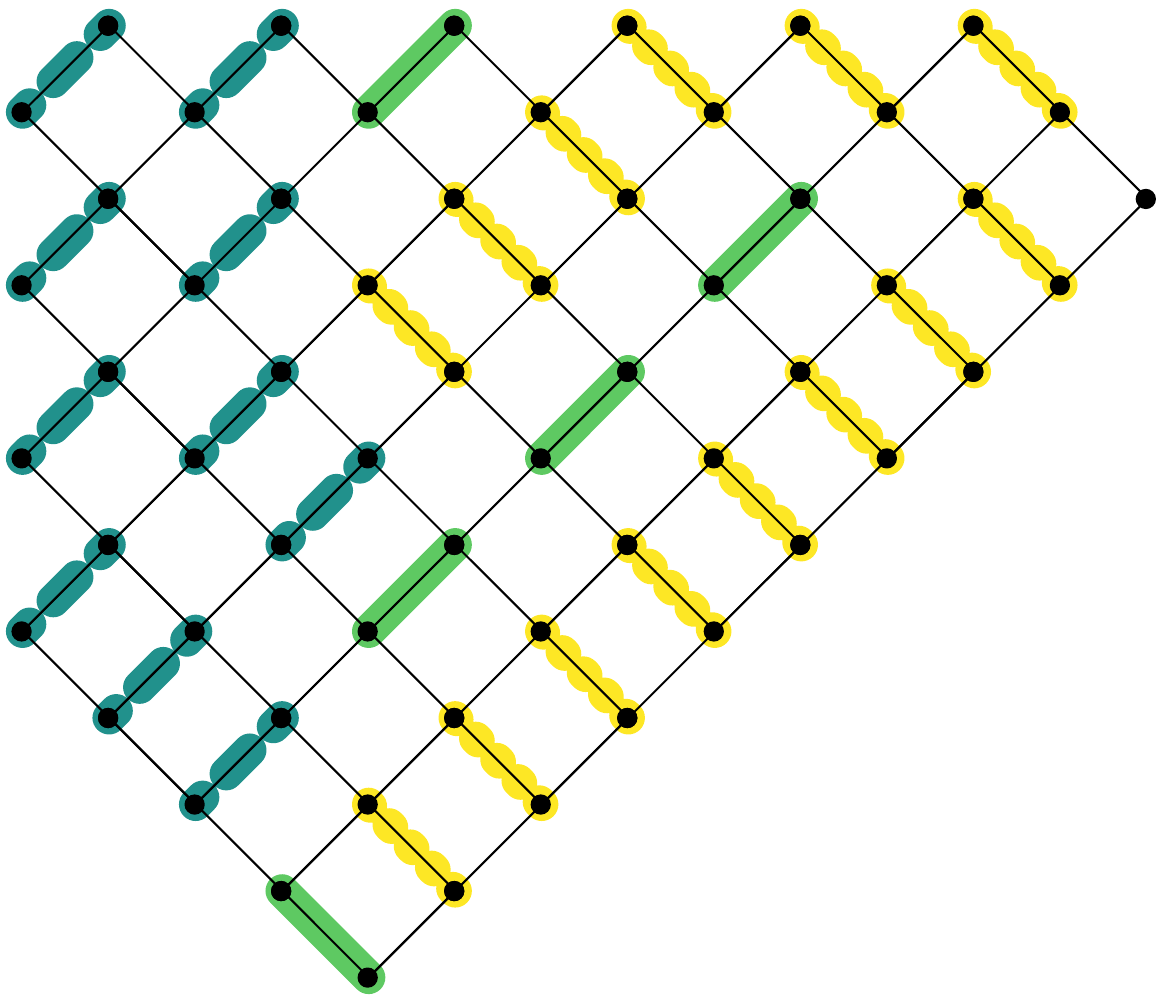}
    \caption{Compressing the $(8,9)$-benzel with an example 3-ribbon tiling to obtain a 2-ribbon tiling, or equivalently a perfect matching, below.}
    \label{fig:suture_2mod3}
\end{figure}

\subsection{Two examples}\label{SubSecAztecTriangle}
We now provide two simple examples of compression, which yield proofs of two conjectures of Propp \cite[Problems 12 and 13]{Propp2022trimer}.
For both examples, we consider tilings of benzels using both stones and two types of bones; we assume without loss of generality that the vertical bone is forbidden. Since both left and right stones are permitted, these tilings after compression correspond to perfect matchings of \cref{fig:0mod3} that use edges yellow and teal edges as well as black edges.

The following result, which resolves Problem 12 of \cite{Propp2022trimer}, states that the number of tilings of the $(n+2,2n+1)$-benzel by left stones, right stones, rising bones, and falling bones is the $n$-th large Schr\"oder number (see \cite[Sequence~A006318]{oeis}). 

\begin{proposition}\label{proposition:problem_12}
    Let $T_n$ denote the number of tilings of the $(n+2,2n+1)$-benzel by left stones, right stones, rising bones, and falling bones (with the convention $T_0=1$). Then
    \begin{equation}\label{eq:problem_12}
        \sum_{n=0}^\infty T_n x^n = \frac{1-x-\sqrt{1-6x+x^2}}{2x}.
    \end{equation}
\end{proposition}
\begin{proof}
    Using the characterization of benzels with $a+b\equiv0\pmod{3}$ from \cref{SecCompress}, we have $s=1$ and $t=n$.
    Using these parameters in the schematic of \cref{fig:0mod3}, we see by applying compression that our stones-and-bones tilings (with the vertical bone forbidden) of the $(n+2,2n+1)$-benzel correspond to perfect matchings of the Aztec triangle of order $n$; see \cref{fig:T_n1} for the case $n=4$, identical to \cite[Figure 13]{Ciucu}. Ciucu \cite[Theorem 4.1]{Ciucu} proved that the number of such perfect matchings is the $n$-th large Schr\"oder number, which completes the proof.
\end{proof} 

\begin{figure}[htbp]
        \centering
        \includegraphics[width=0.145\linewidth]{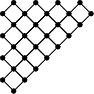}
        \caption{The Aztec triangle of order 4.}
        \label{fig:T_n1}
    \end{figure}

Next we build on \cref{proposition:problem_12} to solve Problem 13 of \cite{Propp2022trimer}.
The relevant enumeration sequence is sequence A006319 in \cite{oeis}.

\begin{proposition}\label{proposition:problem_13}
    Let $T_n'$ denote the number of tilings of the $(n+2,2n)$-benzel by left stones, right stones, rising bones, and falling bones (with the convention $T_0'=1$).
Then
    \begin{equation}\label{eq:problem_13}
        \sum_{n=0}^\infty T_n' x^n = (1-x)\frac{1-x-\sqrt{1-6x+x^2}}{2x}.
    \end{equation}
\end{proposition}
\begin{proof}
    Comparing \eqref{eq:problem_13} with \eqref{eq:problem_12}, we see that it suffices to prove that $T_n'=T_n - T_{n-1}$.
    Using the characterization of benzels with $a+b\equiv2\pmod{3}$ from \cref{SecCompress}, we have $s=1$ and $t=n-1$.
    Using these parameters in the schematic of \cref{fig:2mod3}, we see by applying compression that our stones-and-bones tilings (with the vertical bone forbidden) of the $(n+2,2n)$-benzel correspond to perfect matchings of a shape obtained by removing one ``square'' from the bottom corner of the Aztec triangle of order $n$
    (that is, removing the two lower-left-most vertices and the three edges they participate in).
    See the left pane of \cref{fig:T_n2} for the case $n=4$.
    \begin{figure}[htbp]
        \centering
        \includegraphics[width=0.4646\linewidth]{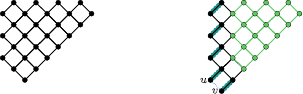}
        \caption{The graph dual to the compressed $(6,8)$-benzel and its relation to the Aztec triangle of order 4.}
        \label{fig:T_n2}
    \end{figure}
    Adding this square back in amounts to adding two new vertices $u$ and $v$, so we see that perfect matchings of this shape correspond to perfect matchings of the Aztec triangle of order $n$ that have these two vertices matched (i.e., perfect matchings that use the blue dotted edge $(u,v)$ in \cref{fig:T_n2}).
    Since the augmented graph (the Aztec triangle of order $n$) has $T_n$ perfect matchings,
    it suffices to show there are $T_{n-1}$ perfect matchings of the Aztec triangle of order $n$ that do not use the edge $(u,v)$.
    After removing pairs of vertices that must be matched to each other (shaded in teal in \cref{fig:T_n2}), the resulting shape (drawn in green in \cref{fig:T_n2}) is the Aztec triangle of order $n-1$.
    Thus, there are indeed $T_{n-1}$ such perfect matchings.
\end{proof}

\section{Decomposition}\label{SecDecomposition}
We now prove \cref{theorem:problem_5}.
Let us fix integers $k\geq 0$ and $n\geq 1$ with $(k,n)\neq(0,1)$, and let $(a,b)=(n+3k,2n+3k-1)$.
As $a+b\equiv2\pmod{3}$, the parameters $s$ and $t$ from \cref{SecCompress} are $s=(2a-b-1)/3=k$ and $t=(2b-a-1)/3=n+k-1$.
As we are only allowed to use right stones, rising bones, and falling bones, compression converts our tiling problem into a perfect matching problem of the form of \cref{fig:2mod3}, where the edges corresponding to left stones in \cref{fig:2mod3} (drawn in yellow) are forbidden; see \cref{fig:decomposition_A} for a visual example of the resulting graph after removing these edges. There are a number of forced matches, as marked in green in \cref{fig:decomposition_A}.
After removing these vertices, we obtain a graph as in \cref{fig:decomposition_B}.

\begin{figure}[htbp]
    \centering
    \includegraphics[width=0.35\linewidth]{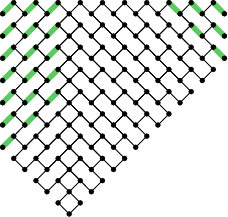}
    \caption{The graph dual to the compressed $(14,18)$-benzel, with edges corresponding to left stones removed.
    Forced matches are marked in green.}
    \label{fig:decomposition_A}
\end{figure}

\begin{figure}[htbp]
    \centering
    \includegraphics[width=0.3158\textwidth]{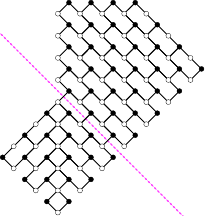}
    \caption{The graph dual to the compressed $(14,18)$-benzel, with forced matches removed. A bipartition of the vertices is indicated via a black-and-white coloring. A pink dotted line separates the graph into two pieces as in \cref{lemma:color_decomposition}.}
    \label{fig:decomposition_B}
\end{figure}

Consider the pink dotted line in \cref{fig:decomposition_B}. By deleting the edges that cross this dotted line, we break the graph into two connected components. Let $A$ be the connected component on the southwest side of the dotted line and $B$ be the connected component on the northeast side; see \cref{fig:decomposition_C}, where $A$ appears on the left and $B$ appears on the right.
\begin{figure}[htbp]
    \begin{center}
    \includegraphics[width=0.1791\textwidth]{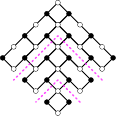} \qquad\qquad\qquad\includegraphics[width=0.2304\linewidth]{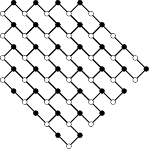}
    \end{center}
    \caption{The two pieces of the compressed $(14,18)$-benzel, with black-and-white bipartitions of the vertices. In the piece on the left, pink dotted lines separate the graph into three $\mathsf{\Lambda}$-shaped pieces to which we can apply \cref{lemma:color_decomposition}.}
    \label{fig:decomposition_C}
\end{figure}
We have colored the vertices black and white in a checkerboard fashion, so that each edge is incident to one black and one white vertex.
Note that $A$ contains as many white vertices as black vertices, and likewise for $B$. All edges joining $A$ and $B$ (i.e., edges that cross the dotted line) connect a white vertex in $A$ to a black vertex in $B$.
Thus, by \cref{lemma:color_decomposition}, the number of perfect matchings of the full graph is the product of the number of perfect matchings of $A$ and the number of perfect matchings of $B$. 

We first enumerate the perfect matchings of $A$.
One can observe that vertically reflecting $A$ yields the same graph as that obtained by applying compression to the shape $\lambda_s$ from \cite{DLPY}, where right stones, i.e., mountain stones, are forbidden; this vertical reflection corresponds to our reflection of \cref{fig:0mod3} to align with \cref{fig:2mod3}.
See \cite[Figure 11]{DLPY} for a visual aid for this compression process (the compression amounts to removing the green columns in that figure).
Working in the uncompressed environment, Defant, Li, Propp, and Young \cite[Proposition 6.2]{DLPY} proved that the number of perfect matchings of $A$ is ${(2s)!!= (2s)(2s-2) \cdots (4)(2) = 2^s s!}$.
Let us briefly sketch the proof of this enumerative result, but in our compressed environment. We can apply \cref{lemma:color_decomposition} to split $A$ along the dotted lines in \cref{fig:decomposition_C}, just as we split the original graph along the dotted line in \cref{fig:decomposition_B}.
(As noted after \cref{lemma:color_decomposition}, this is where, working in the uncompressed environment, the article \cite{DLPY} used a 3-color variant of \cref{lemma:color_decomposition}.)
It then suffices to show that the $i$-th smallest of these ``$\mathsf{\Lambda}$-shapes'' (the article \cite{DLPY} used the term ``$\mathsf{V}$-shape'' in the reflected setting) has $2i$ perfect matchings.
It is straightforward to see that exactly one of the $2i$ edges that are ``perpendicular to the $\mathsf{\Lambda}$'' (i.e., the $2i-2$ edges in the interior and the 2 edges at the two tips of the $\mathsf{\Lambda}$-shape) must be used in a perfect matching, and this choice determines the rest of the perfect matching.

Finally, recalling that $s=k$ and dividing $(2k)!!$ from the expression in \cref{theorem:problem_5}, we find that it remains to show that the number of perfect matchings of $B$ is the quantity
\begin{equation}\label{eq:B1}
    J_{k,n}=\prod_{j=1}^k \frac{(2j-1)!(2j+2n-2)!}{(j+n-1)!(j+n+k-1)!}.
\end{equation}
We can deform the ``brickwork'' pattern of $B$ into a hexagonal grid. There is a well-known bijection between perfect matchings of this hexagonal grid and plane partitions; see \cite{YoungPyramid} for a discussion of this correspondence from brickwork to hexagonal grid to plane partitions.
Specifically, perfect matchings of $B$ correspond to plane partitions of the staircase shape
\[ (t-s,t-s-1,\dots,1)=(n-1,n-2,\dots,1) \]
with parts no larger than $s=k$.
\begin{figure}[htbp]
    \begin{center}
    \includegraphics[width = 0.9\textwidth]{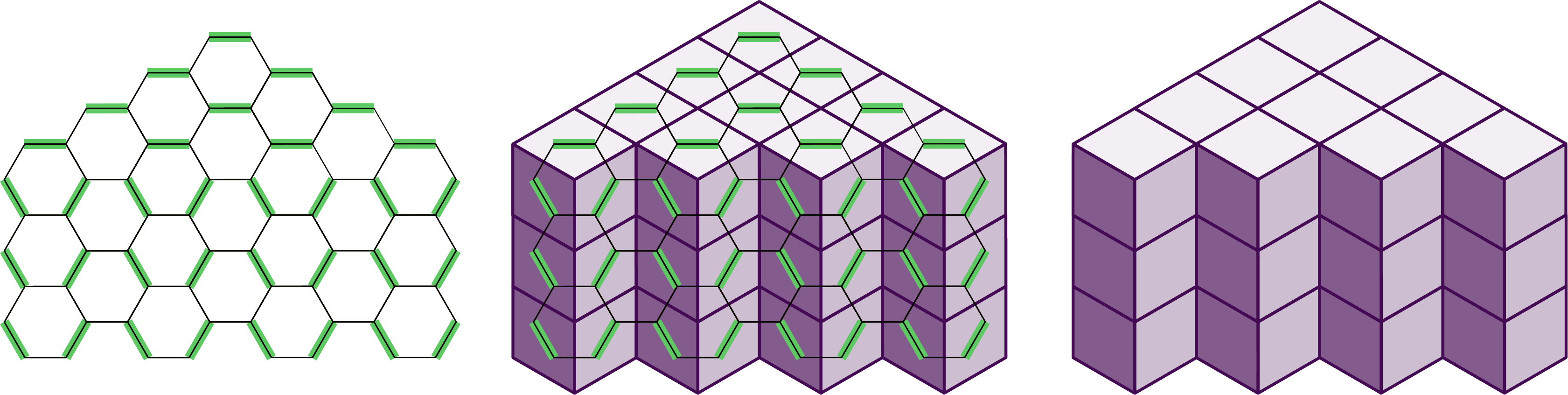}
    \end{center}
    \caption{The $B$ component of the $(14,18)$-benzel, deformed into a hexagonal grid with its maximal matching, and the corresponding plane partition.}
    \label{fig:plane_partition}
\end{figure}
\cref{fig:plane_partition} shows an example of this correspondence for the $(14,18)$-benzel, the same example as in \cref{fig:decomposition_C}; on the left, we show the hexagonal grid obtained by deforming $B$, along with an example matching, in the middle we show how this matching becomes a plane partition, and on the right we show this plane partition by itself.
This matching is the maximal plane partition, which contains the plane partitions corresponding to all of the other matchings.
In our example of the $(14,18)$-benzel, where $s=3$ and $t=7$, this maximal plane partition is the staircase shape $(4,3,2,1)$ with parts all of size 3.
Proctor \cite[Corollary 4.1]{Proctor} found that the number of such plane partitions is the quantity 
\begin{align}\label{eq:B2}
    K_{k,n}=\prod_{i=1}^{n-1}\left(\frac{k+i}{i}\prod_{j=2}^i \frac{2k+i+j-1}{i+j-1}\right)
    &= \prod_{1 \leq i < j \leq n}\frac{2k+i+j-1}{i+j-1}
\end{align}
(see also \cite[Exercise 7.101a]{EC2}). 
To prove that the quantities in \eqref{eq:B1} and \eqref{eq:B2} are equal, one can induct on $k$.
Note that they are trivially equal for $k=0$. Moreover, 
\begin{align*}
    \frac{J_{k,n}}{J_{k-1,n}}=\frac{(2k-1)!(2k+2n-2)!}{(n+2k-1)!(n+2k-2)!},
\end{align*}
while 
\begin{align*}
    \frac{K_{k,n}}{K_{k-1,n}}&=\prod_{1\leq i < j \leq n}\frac{2k+i+j-1}{2k+i+j-3} \\ 
    &= \prod_{i=1}^{n-1} \frac{(2k+i+n-1)!(2k+2i-3)!}{(2k+2i-1)!(2k+i+n-3)!}
    \\ &= \prod_{i=1}^{n-1}\frac{(2k+i+n-1)(2k+i+n-2)}{(2k+2i-1)(2k+2i-2)}
    \\ &= \frac{(2k+2n-2)!(2k+2n-3)!(2k-1)!}{(2k+n-1)!(2k+n-2)!(2k+2n-3)!}
    \\ &= \frac{(2k-1)!(2k+2n-2)!}{(n+2k-1)!(n+2k-2)!}.
\end{align*}
This completes the proof by induction and concludes the proof of \cref{theorem:problem_5}.

\section{Conclusion and Open Problems}\label{SecConclusion}
Propp \cite[Problems 8 to 13]{Propp2022trimer} stated a number of open questions concerning the number of stones-and-bones tilings of benzels that use all stones and bones except the vertical bone.
\cref{proposition:problem_12,proposition:problem_13} addressed Problems 12 and 13.
Using our compression technique, we restate Problems 8 to 11 as perfect matching problems and leave them as open questions, hoping that their simplified form will facilitate further progress with these problems. 

Problem 8 concerns the $(3n,3n)$-benzel. Referencing \cref{fig:0mod3}, we find the graph dual of the compressed $(3n,3n)$-benzel is as depicted in \cref{fig:prob8}, which shows the case $n=3$.
\begin{figure}[htbp]
\centering
\includegraphics[width=0.308\linewidth]{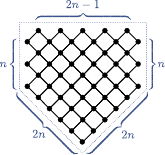}
\caption{The graph dual to the $(9,9)$-benzel.}
\label{fig:prob8}
\end{figure}
\begin{conjecture}[\protect{\cite[Problem 8]{Propp2022trimer}}]
    Let $T_n$ denote the number of stones-and-bones tilings of the $(3n,3n)$-benzel in which the vertical bone is forbidden.
    We have
    \[ \frac{T_n T_{n+2}}{T_{n+1}^2} = \frac{256(2n+3)^2(4n+1)(4n+3)^2(4n+5)}{27(3n+1)(3n+2)^2(3n+4)^2(3n+5)} \]
    for all $n\geq 1$. 
\end{conjecture}

Problem 9 concerns the $(3n+1,3n+1)$-benzel. Referencing \cref{fig:2mod3}, we find the graph dual of the compressed $(3n+1,3n+1)$-benzel is as depicted in \cref{fig:prob9}, which shows the case $n=3$.
\begin{figure}[htbp]
\centering
\includegraphics[width=0.35\linewidth]{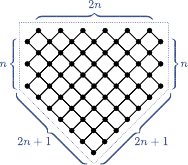}
\caption{The graph dual to the $(10,10)$-benzel.}
\label{fig:prob9}
\end{figure}
Propp \cite{Propp2022trimer} empirically observed that for small $n$, the number of such stones-and-bones tilings had no prime factor greater than or equal to $4n$, suggesting the enumeration could have a nice factored form.
However, obtaining data for large values of $n$ was difficult.
Knowing now how to recast the problem in terms of dimers, we were able to use existing technology for enumerating perfect matchings (specifically the determinant method of Kasteleyn~\cite{Kasteleyn}) to obtain more data. This allows us to supplant Propp's original Problem~9 with the following more precise conjecture. 
\begin{conjecture}[c.f. \protect{\cite[Problem 9]{Propp2022trimer}}]
Let $T_n$ denote the number of stones-and-bones tilings of the $(3n+1,3n+1)$-benzel in which the vertical bone is forbidden.
We have 
\[ \frac{T_n}{T_{n-1}} = \frac{2^{2n}(4n-1)!(4n-2)!n!}{(3n)!(3n-1)!(3n-2)!} \]
for all $n\geq 2$.
\end{conjecture}

Problem 10 yields the pattern depicted in \cref{fig:prob10}, which shows the case $n=3$. Similarly to Problem 8, it also has a conjectured formula for the second quotient.
\begin{figure}[htbp]
\centering
\includegraphics[width=0.388\linewidth]{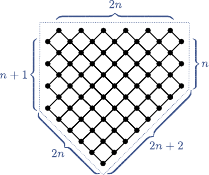}
\caption{The graph dual to the $(10,11)$-benzel.}
\label{fig:prob10}
\end{figure}
\begin{conjecture}[\protect{\cite[Problem 10]{Propp2022trimer}}]
    Let $T_n$ denote the number of stones-and-bones tilings of the $(3n+1,3n+2)$-benzel where the vertical bone is forbidden.
    We have 
    \[ \frac{T_n T_{n+2}}{T_{n+1}^2} = \frac{65536(2n+3)(2n+5)^2(2n+7)(4n+3)(4n+5)^2(4n+7)^2(4n+9)^2(4n+11)}{729(3n+2)(3n+4)^3(3n+5)^2(3n+7)^2(3n+8)^3(3n+10)} \]
    for all $n\geq1$.
\end{conjecture}

Problem 11 yields the pattern depicted in \cref{fig:prob11}, which shows the case $n=3$. As for Problem 9, Propp \cite{Propp2022trimer} empirically observed that for small $n$, the number of such stones-and-bones tilings had no prime factor greater than or equal to $4n$, suggesting the enumeration could have a nice factored form.
As was the case for Problem 9, recasting Problem 11 as a question about perfect matchings allowed us to obtain a conjectural formula.
\begin{figure}[htbp]
\centering
\includegraphics[width=0.347\linewidth]{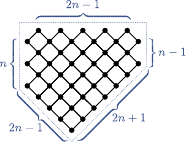}
\caption{The graph dual to the $(8,9)$-benzel.}
\label{fig:prob11}
\end{figure}
\begin{question}[\protect{\cite[Problem 11]{Propp2022trimer}}]
Let $T_n$ denote the number of stones-and-bones tilings of the $(3n-1,3n)$-benzel in which the vertical bone is forbidden.
We have 
\[ \frac{T_n}{T_{n-1}} = \frac{2^{2n-3}(4n-2)!(4n-4)!(n-1)!!(n-3)!!}{(3n-1)!!(3n-2)!(3n-3)!(3n-5)!!}\]
for all $n\geq 2$. 
\end{question}

There are probably more applications of the central idea behind compression. Consider (as a trivial example) the problem of using straight $n$-ominoes (that is, 1-by-$n$ rectangles) to tile an $(n+1)$-by-$(n+1)$ square from which the central $(n-1)$-by-$(n-1)$ square has been removed. One way to prove that there are only two tilings is to note that the $n-1$ middle boxes on each side of the big square must belong to the same $n$-omino and so can be compressed into a single box, reducing the problem to that of using dominoes to tile a 3-by-3 square from which the central box has been removed. In this case, the compression can be achieved geometrically, but all that is required for purposes of enumeration is that it can be done topologically. 

More broadly, given a tiling problem in which certain unions of cells are permitted as tiles, we can look at all nonempty regions that can arise as the intersection of two tiles that actually occur in tilings of the entire region we are trying to tile. In some cases, these will just be the individual cells; in other cases, however, these ``pseudocells'' will be unions of two or more cells, and in this case it may be helpful to imagine identifying all the cells that belong to a common pseudocell. As can be seen in~\cite{DLPY} and the current article, determining when two cells must always be occupied by the same tile in any tiling of a large region can involve non-local arguments that are sensitive to the shape of the boundary of the large region.

\section*{Acknowledgements}
Colin Defant was supported by the National Science Foundation under Award No.\ 2201907 and by a Benjamin Peirce Fellowship at Harvard University. 
James Propp was supported by a Travel Support for Mathematicians gift from the Simons Foundation.

\bibliographystyle{amsinit}
\bibliography{ref}

\end{document}